\documentclass[11pt,reqno]{amsart}
\usepackage{amssymb,latexsym,amsfonts,amsmath}
\usepackage{color}
\usepackage{eucal}
\usepackage[utf8]{inputenc}

\newcommand{\be}{\begin{equation}}
\newcommand{\ee}{\end{equation}}
\newcommand{\bp}{\begin{proof}}
\newcommand{\ep}{\end{proof}}
\newcommand{\bel}{\begin{equation}\label}
\newcommand{\eeq}{\end{equation}}
\newcommand{\bea}{\begin{eqnarray}}
\newcommand{\eea}{\end{eqnarray}}
\newcommand{\bee}{\begin{eqnarray*}}
\newcommand{\eee}{\end{eqnarray*}}
\newcommand{\ben}{\begin{enumerate}}
\newcommand{\een}{\end{enumerate}}

\newcommand{\R}{\mathbb R}

\newcommand{\Z}{\mathbb Z}

\newcommand{\japa}{\langle x\rangle}
\newcommand{\sgn}{{\rm sgn}}
\renewcommand{\u}{\mathfrak{u}}

\newcommand{\N}{\mathbb{N}}

\newcommand{\hil}{\mathcal{H}}

\newcommand{\ji}{\langle}
\newcommand{\jd}{\rangle}

\newtheorem{theorem}{Theorem}[section]
\newtheorem*{theorem-a}{Theorem A}
\newtheorem*{theorem-b}{Theorem B}
\newtheorem{lemma}[theorem]{Lemma}
\newtheorem{proposition}[theorem]{Proposition}

\newtheorem*{remark-a}{Remark A}

\newtheorem{remarks}{Remarks}[section]
\theoremstyle{definition}
\newtheorem{definition}[theorem]{Definition}

\numberwithin{equation}{section}

\title[On properties of solutions to the BO equation]{On decay and  regularity properties of solutions to the Benjamin-Ono equation}
\author{Felipe Linares}
\address[F. Linares]{IMPA\\
Estrada Dona Castorina 110\\
22460-320, Rio de Janeiro, RJ\\Brazil}
\email{linares@impa.br}

\author{Gustavo Ponce}
\address[G. Ponce]{Department  of Mathematics\\
University of California\\
Santa Barbara, CA 93106\\
USA.}
\email{ponce@math.ucsb.edu}


\begin{document}
\thanks{F.L. \!was partially supported by CNPq grant 310329\!/\!2023-0 and FAPERJ grant E\--26\!/\!200465\!/\! 2023. G.P. was partially supported by a Simons Foundation collaboration grant for mathematicians.}.


\begin{abstract} 
 We study  persistence properties of the solution of the Benjamin-Ono equation in weighted Sobolev spaces.  Roughly, we  show that for $\beta<7/2$, the solution $u(x,t)$ of the BO remains in the space $L^2(|x|^{2\beta} dx)$   if and only if its data $u(x,0)$ belongs to this space and it is regular enough, i.e. $u_0\in H^{\beta}(\R)$. 
\end{abstract}
\maketitle

\section{Introduction}

This work is concerned with the Initial Valued Problem (IVP) associated to the Benjamin-Ono (BO) equation
\begin{equation}\label{BO}
\begin{cases}
\partial_tu +\hil\partial_x^2u+u\partial_xu=0,\\
u(x,0)=u_0(x),
\end{cases}
\end{equation}
where $\hil$ denotes the Hilbert transform,
\begin{equation}
\label{hita}
\begin{split}
\hil f(x)&=\frac{1}{\pi} {\rm p.v.}\big(\frac{1}{x}\ast f\big)(x)\\
&=\frac{1}{\pi}\lim_{\epsilon\downarrow 0}\int\limits_{|y|\ge \epsilon} \frac{f(x-y)}{y}\,dy=(-i\,\sgn(\xi) \widehat{f}(\xi))^{\vee}(x).
\end{split}
\end{equation}

 The BO equation was first deduced by Benjamin
\cite{Ben} and Ono \cite{On} as  a model for long internal gravity
waves in deep stratified fluids. Later, it was also shown to be   a
completely integrable system (see \cite{AbSe} and
references therein). In particular, real solutions of \eqref{BO} satisfy
infinitely many conservation laws. These quantities provide an \it a priori \rm
estimate for the $H^{n/2}$-norm, $n\in \Z^+$, of the solution
$\,u=u(x,t)\,$ for \eqref{BO}.

The problem of finding the minimal regularity property, measured in the Sobolev scale $\,H^s(\R)=J^{-s/2}L^2(\R)=(1-\partial_x^2)^{-s/2}L^2(\R),\,s\in\R,$ 
required to guarantee that the IVP \eqref{BO}  is local  or global well-posed  in $\,H^s(\R)$  has been extensively studied, see  \cite{ABFS}, \cite{Io}, \cite{Po}, \cite{KoTz1},  \cite{KeKo}, \cite{Ta},  and \cite{BuPl}. In
\cite{IoKe} global well-posedness was established in $H^s(\R), \;s\geq 0$, see also \cite{MoPi}, \cite{IT} and for a complete review on the BO equation we refer to  \cite{Sa}.  In \cite{MoSaTv} it was proved that no well-posedness for IVP \eqref{BO}  can be established by an argument based only on the contraction principle argument. More recently, in \cite{KiLaVi}, using an argument based on the inverse scattering, it was shown that the IVP  \eqref{BO} is global well posed in $H^s(\R), \;s> -1/2$.

 To describe well-posedness results in weighted Sobolev spaces we introduce the notations:
 $$
Z_{s,r}=H^s(\mathbb R)\cap L^2(|x|^{2r}dx),\hskip10pt
\dot Z_{s,r}=Z_{s,r}\cap \{\,\widehat f(0)=0\},\hskip10pt s,\,r\in\mathbb R.
$$

In this direction, we have:

\begin{theorem-a}[\cite{FP,FLP}]\hskip15pt

\begin{enumerate}

\item[\rm(1)]\label{A-1}
If $u_0\in Z_{5,r},\;\,r\in(0,5/2)$, then  the solution of the IVP for the BO  eq. \eqref{BO} satisfies $\,u\in C([0,T] : Z_{5,r}).$ 
\item[\rm(2)]\label{A-2}
If there exist $\,0<t_1<t_2<T\,$ such that $\,u(\cdot,t_j) \in Z_{5,5/2},\;j=1,2,$ then 
 $\,\;\widehat u(0,t)=0$.
 \item[\rm(3)]\label{A-3}
 If $\,u_0\in \dot Z_{5,r},\,r\in[5/2,7/2)$, then  $\,u\in C([0,T] : \dot Z_{5,r})$.
 \item[\rm(4)]\label{A-4}
 Let $\,u\in C([0,T] : H^5(\mathbb R))$ be a solution of  the BO eq. If there exist $0<t_1<t_2<t_3<T$
 such that  $\,u(\cdot,t_j) \in Z_{5,7/2},\,j=1,2,3,$, then $\,u\equiv 0$.
 \item[\rm(5)]\label{A-5}
  If $\,u_0\in \dot Z_{5,4}\;$ and $\;\int x\,u_0(x)dx\neq 0$, then the solution  in  \eqref{A-3} satisfies 
 \begin{equation*}
\hskip15pt u(\cdot,t^*) \in \dot Z_{5,4},\hskip15pt t^*=-\frac{4}{\|u_0\|^2_2} \int_{\R} xu_0(x)dx.
\end{equation*}

 \end{enumerate}
 \end{theorem-a}
 
 \begin{remark-a}\hskip15pt
 \begin{enumerate}
 \item[\rm(1)]
 
Theorem A part \rm{(1)}  with $r\in (0,2]$, part \rm{(2)}  and \rm{(3)} with $r=3$, and part \rm{(4)} with $r=4$ were previously established in \cite{Io}.
 
 \item[\rm(2)]
The limited decay exhibited on solutions of the BO equation is due to the lack of regularity of the symbol of the operator $\sigma(\hil \partial_x^2)=i \sgn(\xi)\xi^2$ describing the dispersive relation. In particular, the solitons of the BO equation are of the form
 $$
 \eta(x)=\frac{4}{1+x^2},\;\;\;\;\;\;\eta_c(x-t)=c\eta (c(x-ct)),\;\;\;\;c>0.
 $$
 
 \item[\rm(3)] Part  \rm{(5)} of Theorem A tells us that the result described in part \rm{(4)} does not hold if one reduces the hypothesis at just two times.
 \end{enumerate}

 \end{remark-a}

\subsection{Main Results}

\begin{theorem}\label{thm1} \hskip15pt

\begin{enumerate}
\item[\rm(i)] Let $u\in C([0,T]:L^2(\R))\cap L^4([0,T]:L^2(\R))$ be the solution of the IVP for the BO  eq. \eqref{BO} obtained in \cite{MoPi}. If there exist $t_1, t_2\in [0,T]$, $t_1<t_2$, and $\alpha\in(0,1/2]$ such that
\begin{equation}\label{decay-a}
|x|^{\alpha} u(x,t_j)\in L^2(\R), \hskip10pt j=1,2.
\end{equation}
Then
\begin{equation}
u\in C([0,T]:H^{\alpha}(\R)).
\end{equation}

\item[\rm(ii)] Let $u\in C([0,T]:H^s(\R))$, $s>3/2$ be the solution of the IVP for the BO  eq. \eqref{BO}. If there exist $t_1, t_2\in [0,T]$, $t_1<t_2$, and $\alpha\in(s,7/2)$ such that
\begin{equation}\label{decay}
|x|^{\alpha} u(x,t_j)\in L^2(\R), \hskip10pt j=1,2.
\end{equation}
Then
\begin{equation}
u\in C([0,T]:H^{\alpha}(\R)).
\end{equation}
\end{enumerate}
\end{theorem}

\begin{remarks}\hskip10pt

\begin{enumerate} 
\item[\rm(1)] Roughly, Theorem \ref{thm1} affirms that for $\beta<7/2$, the $L^2(|x|^{2\beta} dx)$-norm of the solution remains bounded  in the existence time interval  if and only if the data is regular enough $u_0\in H^{\beta}(\R)$. This kind of result is characteristic of the dispersive equation. In particular, it was proven in \cite{IsLiPo} for the generalized Korteweg-de Vries (gKdV) equation and in \cite{LP} for the intermediate long wave (ILW) equation.  The proof combines an iterative argument with  weighted energy estimates which are reminiscence of the Kato smoothing effect \cite{Ka}. Nevertheless, in these cases, the gKdV and the ILW eqs,  the proofs are significantly simpler. For the KdV the smoothing effect provides a gain of one derivative (local operator), and in the case of the ILW this smoothing effect is described by a smooth pseudo-differential symbol of order $1/2$. However, for the BO equation the gain of $1/2$ derivatives involves an operator with a non-smooth symbol. This highly 
complicates the necessary commutative relations between  fractional weights and  fractional non-smooth operator. In addition, the range of the possible decay is {\rm a priori} limited.

\item[\rm(2)] The  gap $s\in(1/2,3/2]$ between the results in parts {\rm(i)} and {\rm(ii)} of Theorem \ref{thm1} may  seem  to be a consequence of the  arguments used in their proofs. However,  a similar  gap appears  in the following problem : consider the IVP associated to the Burgers equation
\begin{equation}\label{burgers}
\begin{cases}
\partial_tu+u\partial_xu=0,\\
u(x,0)=u_0(x).
\end{cases}
\end{equation}
Question : Given $u_0\in \mathcal H^s(\R)$ for which $s\ge0$ one has an estimate of the form
\begin{equation*}
\sup_{[0,T]} \|u(t)\|_{s,2}\le c_0 \|u_0\|_{s,2} \hskip15pt\text{for some}\hskip5pt T>0?
\end{equation*}

The answer is only for $s=0$ and $s>\frac32$.

\item[\rm(3)] The upper bound $s=\frac72$ in Theorem \ref{thm1} part {\rm(ii)} is related to previous results.

\item[\rm(4)] The results in Theorem \ref{thm1} extend to the generalized BO equation, i.e. higher order non-linearity are considered,  in the range where  
well-posedness is known.

\end{enumerate}
\end{remarks}

\section{Preliminaries}
In this section we will describe several inequalities and identities useful in our analysis.

\subsection{Identities}

\begin{lemma}
Let $f$ be a suitable real function. The following identities hold
\begin{align}
x\hil\partial_x^2 f&= \hil\partial_x^2(xf)+2\hil\partial_xf; \label{Id1}\\
x\hil\partial_x f &= \hil\partial_x(xf)+\hil f; \label{Id2}\\
xD^{1/2}f&=D^{1/2}(xf)+\frac12D^{-1/2}\hil f;\label{Id3}\\
xD^{3/2}f&=D^{3/2}(xf)+\frac32D^{1/2}\hil f;\label{Id4}\\
x^2\hil\partial_x^2 f&= \hil\partial_x^2(x^2f)+4\hil\partial_x(xf)+2\hil f;\label{Id1b}\\
x^3\hil\partial_x^2 f&= \hil\partial_x^2(x^3f)+6\hil\partial_x(x^2f)+6\hil (xf)+\widehat{f}(0);\label{Id1c}\\
[\hil ;x]f&=0\;\;\;\;\;\Leftrightarrow \;\;\;\int f(x)dx=\hat{f}(0)=0;\label{Id1d}.
\end{align}
\end{lemma}

\subsection{Inequalities}

\begin{lemma}\label{CC}
For any $j, k \in \mathbb N\cup\{0\},\,j+k\geq 1$ and any $p\in(1,\infty)$
 \begin{equation} \label{I1}
 \| \partial_x^j \big[\mathcal H;\eta\big]\partial_x^kf\|_p\leq c_{p,j,k}\|\eta^{(j+k)}\|_{\infty} \|f\|_p.
 \end{equation}
 \end{lemma}
 The case $j+k=1$ corresponds to Calder\'on first commutator estimate \cite{Ca}. The general case of \eqref{I1} was established in \cite{BCo}.  For a different proof see \cite{DaGaPo}.

 \begin{lemma}[\cite{Li}]
 If $\alpha\in[0,1]$, $\beta\in (0,1-\alpha]$, and $1<p<\infty$, then
 \begin{equation}\label{I2}
 \|D^{\alpha}[D^{\beta};\eta] D^{1-(\alpha+\beta)}f\|_p\le c\|\partial_x\eta\|_{\infty}\|f\|_p.
 \end{equation}
 \end{lemma}
 
For the proof see Proposition 3.10 in \cite{Li}.

\begin{lemma} [\cite{Li}]
For $\alpha\in (0,1)$ and $p\in (1, \infty)$ one has that
\begin{equation}\label{I6a} 
\|[D^{\alpha};\eta]f\|_p\le c\|D^{\alpha}\eta\|_{\infty}  \|f\|_p,
\end{equation}
and 
\begin{equation}\label{I6b}
\|[D^{\alpha};\eta]f\|_p\le c \|\eta\|_{\infty}\|D^{\alpha}f\|_p.
\end{equation}
\end{lemma}

\begin{lemma}\label{lem1}
 Let $p\in[1,\infty]$ and $s\geq 0$, with $s$ not an odd integer in the case $p=1$. Then 
\begin{equation}\label{I5a}
\|J^{s} (fg)\|_p \leq c (\|f\|_{p_1}\|J^{s}g\|_{p_2} +\|g\|_{q_1}\|J^{s}f\|_{q_2}),
\end{equation}
and
\begin{equation}\label{I5b}
\|D^{s} (fg)\|_p \leq c (\|f\|_{p_1}\|D^{s}g\|_{p_2} +\|g\|_{q_1}\|D^{s}f\|_{q_2})
\end{equation}
with $1/p_1+1/p_2=1/q_1+1/q_2=1/p$.
\end{lemma}

Lemma \ref{lem1} with $p\in(1,\infty)$ was proved in \cite{GO}.
The case $p=p_1=p_2=q_1=q_2=\infty$ was established in \cite{BoLi}. 
The case $p=p_2=q_2=1$ was settled in \cite{OW}. For earlier versions of this result see \cite{KaPo} and \cite{KPV93}.

\begin{lemma}[\cite{LP}]\label{lem2-interpolation} Let $p\in (1,\infty)$. Let $a,b,c,d\in\R$. Then for $\theta\in[0,1]$ 
\begin{equation}\label{interpol-complete}
\|\japa^{\theta a+(1-\theta)c}J^{\theta b+(1-\theta)d} f\|_p \le c_p\|\japa^{a}J^{b} f\|_p^{\theta}\|\japa^{c}J^{d} f\|_p^{(1-\theta)}.
\end{equation}
\end{lemma}

Along the proof of Theorem \ref{thm1} we will use the following result:

\begin{lemma}\label{lem00} For any $a\in (0,1/2)$
\begin{equation}
\label{imp-1}
D^{1/2} \japa^a\in L^{\infty}(\R).
\end{equation}
Moreover, for the truncated version $\japa_N$ of $\japa$  defined in \eqref{we-N} below it follows that for any $a\in (0,1/2)$ there exists $C_a$ independent of $N$ such that
\begin{equation}
\label{imp-2}
\|D^{1/2} \japa^a_N\|_{\infty}\leq C_a.
\end{equation}
\end{lemma}

It is clear that the result in \eqref{imp-1} extends to all $a,b \in (0,1) $ with $b$ instead of $1/2$ and $b>a$.

\begin{proof} For $a \in (0,1)$ one has that $|x|^a,\,\japa^a\in \mathcal S'(\R)$. By homogeneity, it follows that
$$
\widehat{|x|^a}(\xi)=c_1 |\xi|^{-(1+a)},\;\;\;\;\;\widehat{D^{1/2} |x|^a}(\xi)=c_2 |\xi|^{-(1/2+a)}
$$
and
$$
D^{1/2} |x|^a=c_3 |x|^{a-1/2}.
$$
Observing that
$$
\lambda_a(x)=\japa^a-|x|^a=
\begin{cases} 
\begin{aligned}
&\frac{a}{2} |x|^{a-2}+O(|x|^{a-4}),\;\;\;\;\;|x|\to \infty,\\
&1-|x|^a,\hskip69pt \;\;|x|\to 0,
\end{aligned}
\end{cases}
$$
one has that $\lambda_a \in L^1(\R)$ with
$$
D^{1/2}\lambda_a(x)=
\begin{cases} 
\begin{aligned}
& c\, |x|^{a-5/2},\;\;\;\;\;\;\;\;\;\;\;\;\;\;\;\;|x|\to \infty,\\
&-c_3 |x|^{a-1/2},\hskip30pt \; |x|\to 0,
\end{aligned}
\end{cases}
$$
which yields \eqref{imp-1}. Now using the formula 
\begin{equation}
\label{def-0}
D^{1/2}g(x)=c\; \text{p.v.}\int\frac{g(x)-g(y)}{|x-y|^{3/2}}\,dy,\;\;\;\;\;\;\text{p.v.=principal value},
\end{equation}
we can estimate $D^{1/2} \japa^a_N$ and \eqref{imp-2} follows after some simple computations. In fact, the formula \eqref{def-0} can be used to give a direct proof of \eqref{imp-1}.
\end{proof}

 \smallskip

Next, we will recall the definition of the $A_p$ condition. We shall restrict here to
the 1-dimensional case $\R$ (see \cite{M})
\begin{definition} A non-negative function $w\in L^1_{\rm loc}(\R)$ satisfies the $A_p$ inequality with $1<p<\infty$
if
\begin{equation}\label{Ap}
\underset {Q\,\,\rm{interval}}{\sup} \Big(\frac{1}{|Q|}\int_Q w\Big)\Big(\frac{1}{|Q|}\int_Q w^{1-p'}\Big)^{p-1}=c_p(w)<\infty,
\end{equation}
where $1/p+1/p'=1$.
\end{definition}

\begin{lemma}[\cite{HMW}]
The condition \eqref{Ap} is necessary and sufficient for the boundedness of the
 Hilbert transform $\hil$ in $L^p(w(x)\,dx)$, i.e.
  \begin{equation}\label{I3} 
( \int_{\R} |\hil f|^p \,w(x)\,dx)^{1/p} \le c_p^{*}(\int_{\R} |f|^p \,w(x)\,dx)^{1/p}.
 \end{equation}
\end{lemma}

In the case $p = 2$, a previous characterization of $w$ in \eqref{I3} was found in \cite{HS}.
In this case, one has 
\begin{equation}
\label{A_2-precise}
\langle x\rangle^{\beta},|x|^\beta\in A_2\;\;\;\;\Leftrightarrow \;\;\;\beta\in (-1/2,1/2)
\end{equation}
In \cite{Pe} it was shown that the estimate \eqref{I3} holds with $c_2^{*}\leq c\,c_2(w)$ for some fix $c=c(2)$. This can be used to show that all our estimates will not depend on the truncation of the weights.

\subsection{General Argument}

In this subsection we prove a result 
which is employed several times along  our analysis.

\begin{proposition}[General Argument]\label{GENARG} If $\phi\in C^2(\R)$, $\phi''\in L^{\infty}(\R)$, then 
\begin{equation}\label{A1}
\int\hil\partial_x^2v(x)\, v(x)\phi(x)\,dx=-\int (D^{1/2}v(x))^2\phi'(x)\,dx +M_0(t)
\end{equation}
with
\begin{equation}\label{A2}
|M_0(t)|\le c\|\phi''\|_{\infty}\|v\|_2^2.
\end{equation}
\end{proposition}

\begin{proof} We write
\begin{equation*}
\begin{split}
\int \hil\partial_x^2v v\phi&= -\int \hil \partial_x v \partial_xv \phi-\int \hil\partial_x v v \phi'\\
&=M_1(t)+M_2(t).
\end{split}
\end{equation*}

On one hand, we get
\begin{equation}
\label{001}
\begin{split}
M_1(t) &=\int\partial_xv \hil(\partial_xv\phi) = \int \partial_x v [\hil;\phi]\partial_x v +\int  \partial_x v\,\hil\partial_x v\, \phi\\
&= -\int v\partial_x [\hil;\phi]\partial_x v-M_1(t).
\end{split}
\end{equation}
Then 
\begin{equation*}
M_1(t)=-\frac12 \int v\partial_x [\hil;\phi]\partial_x v.
\end{equation*}
By using \eqref{I1} we deduce that
\begin{equation*}
|M_1(t)|\le c\|\phi''\|_{\infty} \|v\|_2^2.
\end{equation*}

On the other hand,
\begin{equation*}
\begin{split}
M_2(t)&=-\int Dv v\phi'= -\int D^{1/2}v D^{1/2}(v\phi')\\
&=-\int D^{1/2}vD^{1/2}v \,\phi'  -\int D^{1/2}v[D^{1/2};\phi']v\\
&= -\int \big(D^{1/2}v\big)^2 \phi' -\int vD^{1/2}[D^{1/2};\phi']v.
\end{split}
\end{equation*}
Applying inequality \eqref{I2} it follows that
\begin{equation*}
\| D^{1/2}[D^{1/2};\phi']v\|_2 \le c\|\phi''\|_{\infty}\|v\|_2.
\end{equation*}

Collecting the above results one gets \eqref{A1}.
\end{proof}

\section{Proof of Theorem \ref{thm1}}

\begin{proof}

First, we consider part (ii). By previous result, see Therorem A, and hypothesis  we already know that
\begin{equation}\label{e0}
|x|^{s}u\in L^{\infty}([0,T]:L^2(\R))\;\;\;\;\text{and}\;\;\;\;u\in C([0,T]:H^s(\R)).
\end{equation}

From  the local  well-posedness  theory one just needs to show 
\begin{equation*}
\exists \,t^{*}\in (t_1,t_2)\subset [0,T] \hskip10pt\text{such that}\hskip5pt u(\cdot,t^{*})\in H^{\alpha}(\R).
\end{equation*}

We shall consider 5 cases:
\begin{itemize}
\item {\bf Case $1$}: $\frac32<s<\alpha<2$,\\
\item {\bf Case $2$}: $\alpha= 2$,\\
\item {\bf Case $3$}: $2<\alpha<\frac{5}{2}$,\\
\item {\bf Case $4$}: $\frac{5}{2}\leq \alpha<3$, \\
\item {\bf Case $5$}: $3\leq \alpha<\frac{7}{2}.$
\end{itemize}

We shall carry out the details for {\bf Case 1},  $3/2<s<\alpha<2$, and sketch the main steps in the other
cases.

For $N\in \mathbb N$ we define
\begin{equation}
\label{we-N}
\ji x\jd_N=
\begin{cases}
\begin{aligned}
&\,\ji x\jd,\,\;\;\;\;\;\;\;\;\;\;\;|x|\leq N,\\
&\,3N/2,\;\;\;\;\;\;\;\;|x|\geq 2N,
\end{aligned}
\end{cases}
\end{equation}
with $\ji x\jd_N\in C^{\infty}(\R)$, even, nondecreasing for $x>0$, and  such that there exists $\,C^*>0$ independent of $N$  such that
\begin{equation}
\label{we-1}
\|\ji x\jd_N' \|_{\infty}\leq C^*.
\end{equation}
Thus,
$$
\partial_x^j \ji x\jd_N \in C^{\infty}_0(\R),\;\;\;\;\forall\,j\in \N,
$$
and a simple computation shows that for $k\in \N$ and $m\geq 0$
\begin{equation}\label{we-3}
|\partial_x^k (\ji x\jd_N^m)|\leq \frac{ C^*}{\; \ji x\jd^{(k-m)}},
\end{equation}
with $C^*=C^*(k,m)$ independent of $N$.

\subsection*{Case 1} : $3/2<s<\alpha\leq 2$.

In this case we divide the argument of proof into four steps.

\smallskip

\noindent \underline{\bf Step 1}.  Using the identity \eqref{Id1} one has that our solution $u$ satisfies
\begin{equation}\label{e1}
\partial_t (xu)+\hil\partial_x^2(xu)+2\hil\partial_xu+xu\partial_xu=0.
\end{equation}
 Hence multiplying \eqref{e1} by $xu\japa_N^{2\theta}x$ where $s=3/2+\theta$, $0<\theta\le 1/2$ one has
\begin{equation}\label{e2}
\begin{split}
&\frac12\frac{d}{dt}\int (xu)^2 \japa_{_N}^{2\theta}x\,dx +\int \hil\partial_x^2(xu) xu \japa_{_N}^{2\theta}x\,dx\\
&+2\int \hil\partial_x u (xu)\japa_{_N}^{2\theta}x\,dx+\int xu \partial_xu\, xu \japa_{_N}^{2\theta}x\,dx\\
&=\frac{d}{dt}E_1(t)+L_1(t)+F_1(t)+N_1(t)=0.
\end{split}
\end{equation}

By hypothesis after integrating in the time interval $[t_1,t_2]$ one has
\begin{equation}\label{e3}
\big|\int_{t_1}^{t_2} \frac{d}{dt}E_1(t)\,dt\big|\le C,
\end{equation}
where from now on $C$ is a constant independent of $N$.

Integration by parts leads to
\begin{equation*}
N_1(t)=-\frac13\int u^3\partial_x(x^3 \japa_{_N}^{2\theta})\,dx.
\end{equation*}
Thus, since $\theta\le 1/2$ by hypothesis \eqref{e0} it follows that
\begin{equation*}
|N_1(t)|\le c_0\|u\|_{\infty}\|\japa^{3/2}u\|_2^2 <C.
\end{equation*}

Using the general argument (Proposition \ref{GENARG}) with $v=xu$ and $\phi(x)=\japa_{_N}^{2\theta}x$ one gets that
\begin{equation*}
L_1(t)=-\int D^{1/2}(xu) D^{1/2}(xu)\,(\japa_{_N}^{2\theta}x)'\,dx +L_{11}(t)
\end{equation*}
with
\begin{equation*}
|L_{11}(t)|\le c\|(\japa_{_N}^{2\theta}x)''\|_{\infty}\|xu\|_2^2<C,
\end{equation*}
using that $\theta\le 1/2$ and \eqref{e0}. Notice that $C$ is independent of $N$.

\smallskip

Finally, we consider $F_1(t)$ and employ \eqref{Id2} to get that
\begin{equation*}
\begin{split}
F_1(t)&=2\int \hil\partial_xu\, xu\, \japa_{_N}^{2\theta}x\,dx\\
&=2\int \hil \partial_x(xu)\,xu\,\japa_{_N}^{2\theta}\,dx+2\int \hil u\,xu\,\japa_{_N}^{2\theta}\,dx\\
&=2\int D^{1/2}(xu) D^{1/2}(xu\,\japa_{_N}^{2\theta})\,dx + \int \japa_{_N}^{\theta}\hil u \japa_{_N}^{\theta}xu\,dx\\
&= 2\int D^{1/2}(xu) D^{1/2}(xu)\,\japa_{_N}^{2\theta}\,dx +2 \int xu D^{1/2} [D^{1/2}; \japa_{_N}^{2\theta}](xu)\,dx\\
&\hskip15pt + 2\int \japa_{_N}^{\theta}\hil u\, \japa_{_N}^{\theta}xu\,dx= F_{11}+F_{12}+F_{13}.
\end{split}
\end{equation*}

Using the inequality \eqref{A_2-precise} and the hypothesis \eqref{e0}
\begin{equation*}
|F_{13}(t)|\le c_0 \|\japa_{_N}^{\theta}\hil u\|_2\|\japa_{_N}^{\theta} xu\|_2< C
\end{equation*}
where $C$ is independent of $N$.

From the inequality \eqref{I2} with $p=2$, $\alpha=\beta=1/2$ it follows that
\begin{equation*}
|F_{12}(t)|\le c\|xu\|_2\|D^{1/2}[D^{1/2}; \japa_{_N}^{\theta}](xu)\|_2\le c_0\|xu\|^2_2\le C.
\end{equation*}

Collecting all the information we see that after integrating in the time interval $[t_1, t_2]$ one has that
\begin{equation*}
\begin{split}
&2\int_{t_1}^{t_2}\int (D^{1/2}(xu))^2 \japa_{_N}^{2\theta}\,dxdt -\int_{t_1}^{t_2}\int (D^{1/2}(xu))^2 (\japa_{_N}^{2\theta}x)'\,dxdt\\
&=\int_{t_1}^{t_2}\int (D^{1/2}(xu))^2 \japa_{_N}^{2\theta}\Big\{1-\frac{x (\japa_{_N}^{2\theta})'}{\japa_{_N}^{2\theta}}\Big\}\,dxdt \le C,
\end{split}
\end{equation*}
$C$ independent of $N$. 

By construction, if $\theta\in (0,1/2)$, then  $1-\dfrac{x(\japa_{_N}^{2\theta})'}{\japa_{_N}^{2\theta}}\ge C'=C'(\theta)>0$. So we conclude that
\begin{equation}
\label{e4a}
\int_{t_1}^{t_2}\int (D^{1/2}(xu))^2 \japa_{_N}^{2\theta}\,dxdt<C
\end{equation}
where $C$ is independent of $N$. Hence by the Monotone Convergence Theorem
\begin{equation}\label{e4}
\int_{t_1}^{t_2}\int (D^{1/2}(xu))^2 \japa^{2\theta}\,dxdt<C.
\end{equation}
In particular,
\begin{equation*}
D^{1/2}(xu) \japa^{\theta}\in L^2(\R) \hskip10pt\text{a.e.} \hskip5pt t\in [t_1,t_2].
\end{equation*}

Reapplying the above argument with $\japa_{_N}^{2\theta+1}$, instead of $\japa^{2\theta}x$ with $\theta\in (0,1/2)$, and using \eqref{e4},
one concludes that
\begin{align}
\japa^{1/2+\theta}xu &\in L^{\infty}([t_1,t_2]:\R), \hskip5pt\text{i.e.} \label{e5}\\
\japa^{\alpha}u &\in L^{\infty}([t_1,t_2]:\R).\label{e6}
\end{align}

We fix $t_3, t_4\in (t_1,t_2)$ with $t_3<t_4$ such that
\begin{equation}\label{e7}
D^{1/2}(xu) \japa^{\theta}(t_j)\in L^2(\R), \hskip10pt j=3,4.
\end{equation}

\noindent\underline{\bf Step 2}. As in the step 1 one needs to consider two cases : $\theta \in (0,1/2)$ and $\theta=1/2$. We shall carry out the details  only in the case $\alpha\in (0,1/2)$. Thus, we multiply the BO equation in \eqref{BO} by $x$ apply the operator $D^{1/2}$ to it, multiply the result by 
$\japa_{_N}^{2\theta-1}x D^{1/2}(xu)$, and use the identity \eqref{Id1}, (to simplify the notation from now on we remove the subindex $N$ in $\langle\cdot\rangle_{_N}$ and remark that all the estimates will hold uniformly on $N$), to get
\begin{equation}\label{e8}
\begin{split}
&\frac12\frac{d}{dt} \int (D^{1/2}(xu))^2\,\japa^{2\theta-1}x\,dx\\
&\hskip10pt+\int \hil\partial_x^2 D^{1/2}(xu) D^{1/2}(xu)\japa^{2\theta-1}x\,dx\\
&\hskip10pt+2\int D^{1/2}\partial_x\hil u D^{1/2}(xu) \japa^{2\theta-1}x\,dx\\
&\hskip10pt+\int D^{1/2}(xu\partial_xu)D^{1/2}(xu) \japa^{2\theta-1}x\,dx\\
&=\frac12 \frac{d}{dt} E_2(t)+L_2(t)+F_2(t)+N_2(t)=0.
\end{split}
\end{equation}

As before from our choice of $t_3, t_4$ we have that 
\begin{equation*}
\Big|\int_{t_3}^{t_4} \frac{d}{dt} E_2(t)\,dt\Big| \le C.
\end{equation*}

Next, we consider $N_2(t)$ written in the case as
\begin{equation*}
N_2(t)=\int D^{1/2}(xu\partial_xu)\japa^{\theta}\,D^{1/2}(xu)\japa^{\theta}\frac{x}{\japa}\,dx
\end{equation*}

We shall employ  that
\begin{equation*}
D^{1/2}(xu\partial_xu)\japa^{\theta}=D^{1/2}(\japa^{\theta}xu\partial_xu)-[D^{1/2};\japa^{\theta}](xu\partial_xu)
\end{equation*}
where by the estimates \eqref{I6a} and \eqref{imp-1}
\begin{equation*}
\begin{split}
\|[D^{1/2};\japa^{\theta}](xu\partial_xu)\|_2&\le c\| D^{1/2}(\japa^{\theta})\|_{\infty}\|xu\partial_xu\|_2\\
&\le c\|xu\|_2\|\partial_xu\|_{\infty} < C.
\end{split}
\end{equation*}
by hypothesis uniformly bounded in the time interval $[0,T]$, and by \eqref{I5b}
\begin{equation*}
\begin{split}
\|D^{1/2}&(\japa^{\theta}xu\partial_xu)\|_2 \\
&\le c\big(\|\partial_xu\|_{\infty} \|D^{1/2}(\japa^{\theta}xu)\|_2
+\|\japa^{\theta}xu\|_{\infty^{-}}\|D^{1/2}\partial_x u\|_{2^{+}}\big)\\
&\le c\|D^{1/2}(\japa^{\theta}xu)\|_2\|u\|_{s,2}.
\end{split}
\end{equation*}
Notice that $\|u\|_{s,2}$ is uniformly bounded in $[0,T]$ and that
\begin{equation*}
D^{1/2}(\japa^{\theta}xu)=\japa^{\theta} D^{1/2}(xu)+[D^{1/2};\japa^{\theta}](xu).
\end{equation*}

Hence from \eqref{e4}, collecting the above argument one has that
\begin{equation*}
\int_{t_3}^{t_4} |N_2(t)|\,dt <C.
\end{equation*}

Next, by the general argument  (Proposition \ref{GENARG} with $v=D^{1/2}(xu)$ and $\phi=\japa^{2\theta-1}x$ one gets that
\begin{equation*}
L_2(t)=-\int (D^{1/2}(D^{1/2}(xu))^2\,(\japa^{2\theta-1}x)'\,dx + L_{21}(t)
\end{equation*}
with
\begin{equation*}
|L_{21}(t)|\le \|(\japa^{2\theta-1}x)''\|_{\infty}\|v\|_2^2\le c_0\|D^{1/2}(xu)\|_2^2.
\end{equation*}
Hence from \eqref{e4}
\begin{equation*}
\int_{t_3}^{t_4} |L_{21}(t)|\,dt <C.
\end{equation*}

Finally, we consider $F_2(t)$ in \eqref{e8}. Using \eqref{Id4}
\begin{equation*}
\begin{split}
F_2(t)&=2\int xD^{1/2}Du\,D^{1/2}(xu) \,\japa^{2\theta-1}\,dx\\
&= 2\int D^{3/2}(xu) D^{1/2}(xu)\japa^{2\theta-1}\,dx+3\int D^{1/2}\hil u D^{1/2}(xu)\japa^{2\theta-1}\,dx\\
&=2\int D(xu) D(xu) \japa^{2\theta-1}\,dx+2\int D(xu)\,[D^{1/2};\japa^{2\theta-1}] D^{1/2}(xu)\,dx\\
&\hskip10pt +3\int D^{1/2}\hil u D^{1/2}(xu)\japa^{2\theta-1}\,dx\\
&= F_{21}(t)+F_{22}(t)+F_{23}(t).
\end{split}
\end{equation*}
where
\begin{equation*}
\begin{split}
F_{22}(t)&= 2\int D(xu)\,[D^{1/2};\japa^{2\theta-1}] D^{1/2}(xu)\,dx\\
&= 2\int D^{1/2}(xu)\,D^{1/2}[D^{1/2};\japa^{2\theta-1}] D^{1/2}(xu)\,dx.
\end{split}
\end{equation*}
Then from inequality \eqref{I2} we deduce 
\begin{equation*}
|F_{22}(t)|\le c_0\|(\japa^{2\theta-1})' \|_{\infty}\|D^{1/2}(xu)\|_2^2
\end{equation*}
and thus
\begin{equation*}
\int_{t_3}^{t_4} |F_{22}(t)|\,dt < C.
\end{equation*}

Also
\begin{equation*}
|F_{23}(t)|\le c\|D^{1/2}u\|_2\|D^{1/2}(xu)\|_2
\end{equation*}
and consequently
\begin{equation*}
\int_{t_3}^{t_4} |F_{23}(t)|\,dt < C.
\end{equation*}

\smallskip

Collecting the above information we can conclude that after integrating in the time interval $[t_3, t_4]$ one gets
\begin{equation*}
\int_{t_3}^{t_4}\int(D(xu))^2\,\japa^{2\theta-1}\, \Big\{ 1-\frac{(\japa^{2\theta-1})'x}{\japa^{2\theta-1}}\Big\}\,dxdt.
\end{equation*}
It is easy to see that $1-\frac{(\japa^{2\theta-1})'x}{\japa^{2\theta-1}}\ge c_0>0$. Hence we conclude that
\begin{equation}\label{e10}
\int_{t_3}^{t_4}\int(D(xu))^2\,\japa^{2\theta-1}\,dxdt<C.
\end{equation}

\smallskip

Now, it follows from  \eqref{I3} that
\begin{equation*}
(D(xu))\,\japa^{\theta-1/2}\in L^2(\R) \hskip10pt\text{implies}\hskip10pt x\japa^{\theta-1/2}\partial_x u\in L^2(\R).
\end{equation*}
Hence we can conclude that
\begin{equation}\label{e10b}
\int_{t_3}^{t_4}\int \big(\japa^{\theta+1/2}\partial_x u\big)^2\,dxdt<C.
\end{equation}

Reapplying the above argument with $\japa^{2\theta}$ instead of $\japa^{2\theta-1}x$ and employing the result \eqref{e10} one deduces
that
\begin{equation*}
D^{1/2}(xu)\japa^{\theta}\in L^{\infty}([t_3,t_4]: L^2(\R)).
\end{equation*}

\smallskip

Next, we fix $t_5, t_6\in  [t_3, t_4]$, $t_5<t_6$, such that
\begin{equation}\label{e10a}
\japa^{\theta+1/2}\partial_xu(x, t_j)\in L^2(\R), \hskip10pt j=5,6.
\end{equation}

\smallskip

\noindent\underline{\bf Step 3}.   As in the previous steps one needs to consider two cases : $\theta \in (0,1/2)$ and $\theta=1/2$. We shall carry out the details  only in the case $\alpha\in (0,1/2)$. Consider  the BO equation after applying the operator $\japa^{2\theta}\partial_x$, multiply the
result by $\japa^{\theta}\partial_xu x$ and integrate the result to get
\begin{equation}\label{e11}
\begin{split}
&\frac12\frac{d}{dt}\int \big(\japa^{\theta}\partial_xu)^2x\,dx+\int \japa^{\theta}\partial_x\hil\partial_x^2u \japa^{\theta}\partial_xu x\,dx\\
&\hskip10pt+\int \japa^{\theta}\partial_x(u\partial_x u)\japa^{\theta}\partial_x u\, x\,dx\\
&=\frac12\frac{d}{dt}E_3(t)+L_3(t)+N_3(t)=0.
\end{split}
\end{equation}

After integrating in the time interval $[t_5, t_6]$, using \eqref{e10a} one has that
\begin{equation*}
\Big|\frac{d}{dt}\int_{t_5}^{t_6} E_3(t)\,dt\Big| < C.
\end{equation*}

Next, to bound $N_3(t)$ we write
\begin{equation*}
\begin{split}
N_3(t)&=\int \japa^{\theta} \partial_x(u\partial_xu)\,\japa^{\theta}\partial_xu\,dx\\
&=\int \japa^{\theta}\partial_xu \partial_xu \japa^{\theta}\partial_x u\, x\,dx+
\int \japa^{\theta}\partial_x^2u\,u\japa^{\theta}\partial_xu\,x\,dx\\
&=\int \japa^{2\theta}x(\partial_x u)^3\,dx +\frac12 \int \partial_x\big((\partial_x u)^2) \japa^{2\theta} xu\,dx\\
&=\frac12 \int \japa^{2\theta}x(\partial_xu)^3-\frac12\int (\partial_x u)^2u (\japa^{2\theta}x)'\,dx\\
&=N_{31}(t)+N_{32}(t).
\end{split}
\end{equation*}

We observe that
\begin{equation*}
|N_{31}(t)|\le c\|\partial_xu\|_{\infty}\|\japa^{\theta+1/2}\partial_xu\|_2
\end{equation*}
which after integrating in the time interval $[t_5, t_6]$ one has from \eqref{e10b} and the hypothesis $s>3/2$ that
\begin{equation*}
\int_{t_5}^{t_6} |N_{31}(t)|\,dt < C.
\end{equation*}
Similarly for $N_{32}(t)$ since
\begin{equation*}
|N_{32}(t)|\le c\|\japa^{2\theta}(\partial_x u)^2u\|_2\le c\|u\|_{\infty}\|\japa^{\theta}\partial_x u\|_2^2.
\end{equation*}

\smallskip

Finally, we consider the contribution of $L_3(t)$. From the general argument, Proposition \ref{GENARG}, with $v=\partial_x u$ and $\phi=\japa^{2\theta}x$,
\begin{equation*}
\begin{split}
L_3(t)&=\int \hil\partial_x^2\partial_xu\,\partial_xu \japa^{2\theta}x\,dx\\
&=-\int (D^{1/2}\partial_xu)^2(\japa^{2\theta}x)'\,dx +L_{31}(t).
\end{split}
\end{equation*}
where
\begin{equation*}
|L_{31}(t)|\le \|(\japa^{2\theta}x)''\|_{\infty}\|\partial_xu\|_2^2<C
\end{equation*}
uniformly in $t$  (recall the hypotheses $\theta<1/2$ and $s>3/2$).

Since $(\japa^{2\theta}x)'\simeq \japa^{2\theta}$ collecting the above information we conclude that
\begin{equation}\label{e15}
\int_{t_5}^{t_6}\int (D^{1/2}\partial_x u)^2 \japa^{2\theta}\,dxdt< C,
\end{equation}
i.e.
\begin{equation}\label{e16}
D^{1/2}\partial_x u \japa^{\theta}\in L^2(\R)\hskip5pt \text{a.e.}\hskip5pt t\in [t_5,t_6].
\end{equation}

Reapplying the argument with $\japa^{2\theta+1}$ instead of $\japa^{2\theta}x$ and using \eqref{e15}
\begin{equation}\label{e17}
\japa^{\theta+1/2}\partial_xu\in L^{\infty}([t_5,t_6]:L^2(\R)).
\end{equation}

\smallskip

We fix $t_7,t_8\in [t_5,t_6]$ such that \eqref{e16} holds for $t_7, t_8$.

\smallskip

\noindent\underline{\bf Step 4}.  As in the step 1 one needs to consider two cases : $\theta \in (0,1/2)$ and $\theta=1/2$. We shall carry out the details  only in the case $\alpha\in (0,1/2)$. We apply the operator $D^{1/2}\partial_x$ to the equation in \eqref{BO}, multiply the result by 
$\japa^{2\theta-1}x D^{1/2}\partial_xu$ and integrate the result to get that
\begin{equation}\label{e18}
\begin{split}
&\frac12\frac{d}{dt}\int (D^{1/2}\partial_x u)^2\japa^{2\theta-1}x\,dx \\
&\hskip10pt +\int D^{1/2}\partial_x\hil\partial_x^2u D^{1/2}\partial_xu\japa^{2\theta-1}x\,dx\\
&\hskip10pt + \int D^{1/2}\partial_x(u\partial_xu)D^{1/2}\partial_xu \japa^{2\theta-1}x\,dx\\
&=\frac12\frac{d}{dt}E_4(t) +L_4(t)+N_4(t)=0.
\end{split}
\end{equation}

\smallskip

From our choice of $t_7, t_8$ one has that
\begin{equation*}
\Big|\int_{t_7}^{t_8} \frac{d}{dt} E_4(t)\,dt\Big| < C.
\end{equation*}

To control the contribution of $N_4(t)$ we write
\begin{equation*}
N_4(t)=\int \japa^{\theta} D^{1/2}\partial_x(u\partial_xu)\japa^{\theta}D^{1/2}\partial_xu\,\frac{x}{\japa}\,dx
\end{equation*}
and use that
\begin{equation*}
\japa^{\theta}D^{1/2}f=D^{1/2}(\japa^{\theta}f)-[D^{1/2};\japa^{\theta}]f.
\end{equation*}
Thus
\begin{equation}\label{e19}
\begin{split}
N_4(t)&=\int D^{1/2}(\japa^{\theta}\partial_x(u\partial_x u))\, D^{1/2}(\japa^{\theta}\partial_xu)\frac{x}{\japa}\,dx\\
&\hskip10pt-\int [D^{1/2};\japa^{\theta}]\partial_x(u\partial_x u)\, D^{1/2}(\japa^{\theta}\partial_xu)\frac{x}{\japa}\,dx\\
&\hskip10pt-\int D^{1/2}(\japa^{\theta}\partial_x(u\partial_xu)\, [D^{1/2};\japa^{\theta}]\partial_x u \frac{x}{\japa}\,dx\\
&\hskip10pt+\int [D^{1/2};\japa^{\theta}]\partial_x(u\partial_xu)\,[D^{1/2};\japa^{\theta}]\partial_x u \frac{x}{\japa}\,dx\\
&=N_{41}+N_{42}+N_{43}+N_{44}.
\end{split}
\end{equation}

Using the estimates \eqref{I2} with $\alpha=1/2$, $\eta=\japa^{\theta}$, and $p=2$ we get the appropriated
bounds for $N_{42}$ and $N_{44}$, that is,
\begin{equation*}
\begin{split}
\|[D^{1/2};\japa^{\theta}]\partial_x(u\partial_x u)\|_2&\le c_0\|D^{1/2}\hil(u\partial_x u)\|_2\\
&\le c_0\|D^{1/2}(u\partial_x u)\|_2\le c\|u\|_{s,2}^2
\end{split}
\end{equation*}
and $\|D^{1/2}(\japa^{\theta}\partial_xu)\|_2$ bounded after integrating in the time interval $[t_7, t_8]$.

\smallskip

To control $N_{43}$ we write
\begin{equation*}
\begin{split}
N_{43}&= \int D^{1/2}((\japa^{\theta})')u\partial_xu)\, [D^{1/2};\japa^{\theta}]\partial_xu \,\frac{x}{\japa}\,dx\\
&\hskip15pt + \int D^{1/2}(\partial_x(\japa^{\theta}u\partial_xu))\,[D^{1/2};\japa^{\theta}]\partial_xu\, \frac{x}{\japa}\,dx\\
&= N_{431}+ N_{432}.
\end{split}
\end{equation*}

A familiar argument yields
\begin{equation*}
N_{431}\le \|D^{1/2}(u\partial_x u)\|_2\|\partial_x u\|_2\le C\|\partial_x u\|_2 \|u\|_{s}^2
\end{equation*}
and
\begin{equation*}
\begin{split} 
N_{432}&= \int D^{1/2}(\japa^{\theta}u\partial_xu)\, \partial_x\big( [D^{1/2};\japa^{\theta}]\partial_xu \, \frac{x}{\japa}\big)\,dx\\
&= \int D^{1/2}(\japa^{\theta}u\partial_xu)\,[D^{1/2};\japa^{\theta}]\partial_xu \big(\frac{x}{\japa}\big)'\,dx\\
&\hskip10pt +\int D^{1/2}(\japa^{\theta}u\partial_x u) \Big\{ [D^{1/2}; (\japa^{\theta})'] \partial_xu\\
&\hskip25pt +[D^{1/2};\japa^{\theta}] D^{1/2}D^{1/2}\hil \partial_xu\Big\}\, \frac{x}{\japa}\,dx\\
&= N_{4321}+ N_{4322}+N_{4323}.
\end{split}
\end{equation*}

Now
\begin{equation}\label{e20}
\begin{split}
&|N_{4321}| \le \|D^{1/2}(\japa^{\theta}u \,\partial_xu)\|_2\|D^{1/2}u\|_2\\
&\le \big(\|[D^{1/2};\japa^{\theta}u]\partial_xu\|_2+\|\japa^{\theta}u D^{1/2}\partial_xu\|_2\big)\|D^{1/2}u\|_2\\
&\le \big(c_0\|\japa^{\theta}u\|_{\infty}\|D^{1/2}\partial_xu\|_2+\|\japa^{\theta}u\|_{\infty}\|D^{1/2}\partial_xu\|_2\big)\|D^{1/2}u\|_2\\
&\le c\|D^{1/2}u\|_2\|u\|_{3/2,2}\|\japa^{\theta}u\|_{1,2} <C
\end{split}
\end{equation}
(see \eqref{e17} and hypothesis $s>3/2$).

The boundedness  of $N_{4322}$ can be obtained using a similar argument. Following the argument leading to \eqref{e20}
yields
\begin{equation*}
|N_{4323}|\le c\|D^{1/2}(\japa^{\theta}u\partial_xu)\|_2\|D^{1/2}\partial_xu\|_2 <C.
\end{equation*}

\smallskip

Finally, we turn to $N_{41}$ in \eqref{e19}.

\begin{equation*}
\begin{split}
N_{41}(t)&=\int D^{1/2}(\japa^{\theta}u\partial_x^2 u) \,D^{1/2}(\japa^{\theta}\partial_xu) \frac{x}{\japa}\,dx\\
&\hskip15pt+\int D^{1/2}(\japa^{\theta}\partial_xu\partial_xu) D^{1/2}(\japa^{\theta}\partial_xu) \frac{x}{\japa}\,dx\\
&=N_{411}+N_{412}
\end{split}
\end{equation*}
where the use of the commutator estimates  lead to
\begin{equation*}
\|D^{1/2}(\japa^{\theta}u\partial_x u)\|_2\le c\|\partial_xu\|_{\infty}\|D^{1/2}(\japa^{\theta}\partial_xu)\|_2.
\end{equation*}
Hence
\begin{equation*}
|N_{412}|\le \|u\|_{s,2}\|D^{1/2}(\japa^{\theta}\partial_xu)\|_2^2\le \|u\|_{s,2}
\|D^{1/2}(\japa^{\theta}\partial_xu)\|_2^2
\end{equation*}
since $\|u\|_{s,2}$ is uniformly bounded in $[0,T]$ after integrating in the time interval $[t_7, t_8]$, using the previous step \eqref{e15} and a familiar inequality we deduce that
\begin{equation*}
\int_{t_7}^{t_8}|N_{412}(t)|\,dt < C.
\end{equation*}

\smallskip

We rewrite $N_{411}$ as
\begin{equation*}
\begin{split}
N_{411}(t)&=\int D^{1/2}(\japa^{\theta}u\partial_x^2 u) \,D^{1/2}(\japa^{\theta}\partial_xu) \frac{x}{\japa}\,dx\\
&=-\int D^{1/2}(u(\japa^{\theta})'\partial_x u)D^{1/2}(\japa^{\theta}\partial_xu) \frac{x}{\japa}\,dx\\
&\hskip10pt +\int D^{1/2}(u\partial_x(\japa^{\theta}\partial_xu))\, D^{1/2}(\japa^{\theta}\partial_xu) \frac{x}{\japa}\,dx\\
&=N_{4111}+N_{4112}.
\end{split}
\end{equation*}

The term $N_{4111}$ is bounded by
\begin{equation*}
\|D^{1/2}(u\partial_xu)\|_2\|D^{1/2}(\japa^{\theta}\partial_xu)\|_2\le c\|u\|_{s,2}^2\|D^{1/2}(\japa^{\theta}\partial_xu)\|_2
\end{equation*}
which after integration in the time interval $[t_7,t_8]$ is uniformly bounded.

\smallskip

It remains to consider
\begin{equation*}
\begin{split}
N_{4112}(t)&= \int D^{1/2}(u\partial_x(\japa^{\theta}\partial_xu))D^{1/2}(\japa^{\theta}\partial_xu)\frac{x}{\japa}\,dx\\
&=\int u \partial_xD^{1/2}(\japa^{\theta}\partial_xu)\, D^{1/2}(\japa^{\theta}\partial_xu)\frac{x}{\japa}\,dx\\
&\hskip10pt+\int [D^{1/2};u]\partial_x(\japa^{\theta}\partial_xu)D^{1/2}(\japa^{\theta}\partial_xu)\frac{x}{\japa}\,dx\\
&=\frac12 \int u \partial_x\big(D^{1/2}(\japa^{\theta}\partial_xu)\big)^2 \frac{x}{\japa}\,dx\\
&\hskip10pt +\int [D^{1/2};u]D^{1/2}\hil D^{1/2}(\japa^{\theta}\partial_xu) D^{1/2}(\japa^{\theta}\partial_xu)\frac{x}{\japa}\,dx\\
&=-\frac12 \int \partial_x\Big(u\frac{x}{\japa}\Big) (D^{1/2}(\japa^{\theta}\partial_xu)\big)^2\,dx\\
&\hskip10pt +\int[D^{1/2};u] D^{1/2}\hil D^{1/2}(\japa^{\theta}\partial_xu) D^{1/2}(\japa^{\theta}\partial_xu)
\frac{x}{\japa}\,dx
\end{split}
\end{equation*}
which is bounded by
\begin{equation*}
\big(\|u\|_{\infty}+\|\partial_x u\|_{\infty}\big)\|D^{1/2}(\japa^{\theta}\partial_xu)\|_2^2
\le \|u\|_{s,2}\|D^{1/2}(\japa^{\theta}\partial_xu)\|_2^2.
\end{equation*}
which after integrating in the time interval is bounded.

\smallskip

It remains to consider $L_4(t)$ in \eqref{e18} which by the general argument, Proposition \ref{GENARG}  with
 $v=D^{1/2}\partial_xu$, $\phi=\japa^{2\theta-1}x$
 \begin{equation*}
 \begin{split}
 L_4(t)&=\int D^{1/2}\partial_x\hil \partial_x^2u D^{1/2}\partial_xu \japa^{2\theta-1}x\,dx\\
&= \int (D\partial_xu)^{2}(\japa^{2\theta-1}x)'\,dx +L_{41}(t).
\end{split}
\end{equation*}
with
\begin{equation*}
|L_{41}(t)|\le c\|(\japa^{2\theta-1}x)''\|_{\infty}\|D^{1/2}\partial_xu\|_2 \le c_0\|u\|_{s,2}.
\end{equation*}
Hence after integrating in $[t_7, t_8]$ one gets
\begin{equation*}
\int_{t_7}^{t_8} \int (D\partial_xu)^{2}(\japa^{2\theta-1}x)'\,dxdt <C.
\end{equation*}
That is,
\begin{equation*}
D\partial_xu \japa^{\theta-1/2}\in L^2(\R) \hskip10pt\text{a.e.}\hskip10pt t\in [t_7,t_8].
\end{equation*}

From the continuity properties of the Hilbert transform ($\hil^2=-I$)
\begin{equation*}
D\partial_xu \japa^{\theta-1/2}=\hil \partial_x^2u\japa^{\theta-1/2}\in L^2(\R).
\end{equation*}
This implies that $\partial_x^2u\japa^{\theta-1/2}\in L^2(\R)$ and $\japa^{\alpha} u\in L^2(\R)$, with $\alpha=3/2+\theta$,
by using interpolation.
Therefore, $\japa^{\alpha-2}J^2u\in L^2(\R)$ and $\japa^{\alpha}u\in L^2(\R)$.

\smallskip

Finally, by interpolation in \eqref{interpol-complete} we conclude  that $J^{\alpha}u\in L^2(\R)$. 
 
 This completes {\bf Case 1}.
\vskip.2in
 
\noindent {\bf Case 2} : $\alpha=2$. It has four steps:

\vskip1mm

 \noindent \underline{\bf Step 1}. Multiplying \eqref{e1} by $x^3u\japa$ and integrating the result  one has
\begin{equation}\label{e=2}
\begin{split}
&\frac12\frac{d}{dt}\int x^3 u^2 \japa\,dx +\int \hil\partial_x^2u\, x^3u\, \japa_{_N}\,dx
+\int u \partial_xu\, x^3u \japa\,dx\\
&=\frac{d}{dt}\tilde E_1(t)+\tilde L_1(t)+\tilde N_1(t)=0.
\end{split}
\end{equation}

The estimates for the terms $\frac{d}{dt}\tilde E_1(t)$ and $\tilde N_1(t)$ are similar to those given above for the case $\theta \in (0,1/2)$
so we just need to handle $\tilde L_1(t)$. Thus, we write
\begin{equation*}
\begin{split}
\tilde L_1(t)&= -\int x \hil\partial_xu\,x\partial_xu\,x \japa\,dx -\int \hil \partial_x u\,u\,(x^3 \japa)'\,dx\\
&=- \tilde L_{11}(t) - \tilde L_{12}(t).
\end{split}
\end{equation*}
 Then, using \eqref{Id2}  one has 
 \begin{equation*}
\begin{split}
-\tilde L_{11}(t)&=-\int(\hil \partial_x(xu)+\hil u)(\partial_x(xu)-u)x\,\japa\,dx\\
&=-\int \hil \partial_x(xu)\,\partial_x(xu)\,x\,\japa\,dx + \int\hil \partial_x(xu)\,ux\,\japa\,dx\\
&\hskip12pt -\int \hil u\,\partial_x(xu)\,x\,\japa\,dx+\int \hil u\, u\,x\,\japa\,dx\\
&= -\tilde L_{111}(t)+\tilde L_{112}(t)-\tilde L_{113}(t)+\tilde L_{114}(t).
\end{split}
\end{equation*}
Now
\begin{equation*}
\begin{split}
-\tilde L_{111}(t)&=\int \partial_x(xu)\,\hil(\partial_x(xu)\,x )\japa\,dx \\
&=\int \partial_x(xu)\,\hil\partial_x(xu)\,x\,\japa\,dx + \int \partial_x(xu)\,[\hil;x\,\japa]\partial_x(xu)\,dx\\
&=\tilde L_{111}(t)-\int xu\,\partial_x[\hil;x\,\japa]\partial_x(xu)\,dx\\
\end{split}
\end{equation*}
thus by \eqref{I1}
$$
|\tilde L_{111}(t)|\leq C \|xu\|_2^2,\;\;\;\;\;\;\;\;C\;\;\text{independent of}\;\;\;N.
$$
Next, we write
$$
\tilde L_{112}(t)=\int D(xu)\,xu\,\japa\,dx 
$$ 
and by integration by parts and \eqref{Id2}
\begin{equation*}
\begin{split}
-\tilde L_{113}(t)&=\int \hil \partial_xu\,(xu)\,x\japa\,dx+\int \hil u\,xu\,(x\japa)'\,dx\\
&=\int D(xu)\,xu\,\japa\,dx+\int \hil u\,xu\,(x(\japa)'+\japa)\,dx
\end{split}
\end{equation*}
Noticing that from \eqref{A_2-precise} if $2=s+\gamma, \gamma\in (0,1/2)$ one has 
$$
|\int \hil u \,u\,x\japa\,dx|\leq \| \japa^{\gamma} \hil u\|_2 \|\japa^s u\|_2\leq  \|\japa^s u\|_2^2\leq C,
$$
one has that
$$
-\tilde L_{113}(t)=\int D(xu)\,xu\,\japa\,dx + N_1(t),
$$
with 
$$
|N_1(t)|+|\tilde L_{114}(t)|\leq C\;\;\;\;\;\;\;\;C\;\;\text{independent of}\;\;\;N.
$$
Finally, one sees by a familiar argument that
$$
-\tilde L_{12}(t)=-\int D(xu)\,xu\,(x^3\japa)'/x^2\,dx + N_2(t),
$$
with $|N_2(t)| $ uniformly bounded in the time interval $[0,T]$.

Hence, collecting the above information one sees that
$$
\tilde L_1(t)= -\int D(xu)\,xu\,\big((x^3 \japa)'/x^2-2\japa\big)\,dx+N_3(t),
$$
with $|N_3(t)| $ uniformly bounded in the time interval $[0,T]$. Since
$$
\big((x^3 \japa)'/x^2-2\japa\big)\geq 1,
$$
we have the same result as in \eqref{e4a}-\eqref{e6} with $\theta=1/2$.

The reminder three steps in the proof of {\bf Case 2} are similar to those described above, so it will be omitted.
\medskip

 \vskip.2in
 
\noindent {\bf Case 3} : $2<\alpha<2+1/2$. It has five steps.

\noindent\underline{\bf Step 1}. By  {\bf Case 2}   we know that
\begin{equation}
\label{e00}
|x|^{2}u\in L^{\infty}([0,T]:L^2(\R))\;\;\;\;\text{and}\;\;\;\;u\in C([0,T]:H^2(\R)).
\end{equation}
and by hypothesis  $ |x|^{2+\theta}u(\cdot, t_j)\in L^2(\R),\; j=1,2$ with $\alpha=2+\theta,\,\theta\in(0,1/2)$.

\medskip

We multiply the BO equation by $x^2 \japa^{2\theta-1}xu$ and integrate the result to get
\begin{equation*}
\frac12\frac{d}{dt}\int\!\! x^4 u^2 \japa^{2\theta-1}xdx +\int\! \hil\partial_xu\, u\,x^4 \japa^{2\theta-1}xdx
+\int\! \u \partial_xu\,u\,x^4 \japa^{2\theta-1}dx=0.
\end{equation*}

From \eqref{e00} and hypothesis we have that after integrating in the time interval $[t_1,t_2]$ the contribution of the first and third terms above are bounded for $\theta\in(0,1/2)$. So it remains to consider the second one. Using \eqref{Id1b} we write
$$
\begin{aligned}
&\int \hil\partial_xu\, u\,x^4 \japa^{2\theta-1}xdx=
\int \hil\partial_x(x^2u)\, x^2u\, \japa^{2\theta-1}xdx\\
&+4\int \hil\partial_x(xu)\, x^2u\, \japa^{2\theta-1}xdx
+2\int \japa^{2\theta-1} x \hil u\, u\,x^2 dx\\
&=L_1(t)+L_2(t)+L_3(t).
\end{aligned}
$$

A familiar argument shows that the contribution of the terms $L_1(t)$ and $L_2(t)$ can be written by a bounded term plus a multiple of
$$
\int (D^{1/2}(x^2u))^2\,\japa^{2\theta-1}\,dx.
$$
To handle $L_3(t)$ we write:
\begin{equation}
\label{term-a}
|L_2(t)|\leq c\|\japa^{2\theta-1} x \hil u\|_2\| u\,x^2\|_2,
\end{equation}
which is bounded only for $\theta\in (0,1/4)$, see \eqref{A_2-precise} and \eqref{e00}. Collecting the above information and following a familiar argument we can conclude that
\begin{equation}
\label{term-b}
|x|^{2+\beta}u\in L^{\infty}([t_1,t_2]:L^2(\R)),\;\;\beta\in(0,1/4).
\end{equation}
Reapplying the above argument using \eqref{term-b} instead of \eqref{e00}. one obtains that the term in \eqref{term-a} 
is bounded for $\theta\in (0,3/8)$. This iteration allows to extends the result to $\alpha\in(0,1/2)$. Notice that by Theorem A part (2) the result for $\alpha=1/2$ is false without the hypothesis $\int u_0(x)\,dx=0$.

The rest of the proof in this case follows a similar argument to the one provided in details in the {\bf Case 1}, so it will be omitted.

\medskip

\noindent{\bf Case 4} : $2+1/2\leq\alpha\leq 3$. It has six steps.

In this case we apply for $xu(x,t)$ the argument given in the {\bf Case 1} for $u(x,t)$.

\medskip

\noindent{\bf Case 5} : $3<\alpha<3+1/2$. It has seven steps.

\noindent\underline{\bf Step 1}. By  {\bf Case 4}  we have that
\begin{equation}
\label{e000}
|x|^{3}u\in L^{\infty}([0,T]:L^2(\R))\;\;\;\;\text{and}\;\;\;\;u\in C([0,T]:H^3(\R)).
\end{equation}
and by hypothesis  $ |x|^{3+\theta}u(\cdot, t_j)\in L^2(\R),\; j=1,2$ with $\alpha=3+\theta,\,\theta\in(0,1/2)$.

\medskip

We multiply the BO equation by $x^3 \japa^{2\theta-1}xu$ and integrate the result to get
$$
\frac12\frac{d}{dt}\int \!\!x^6 u^2 \japa^{2\theta-1}xdx +\int \!\hil\partial_x^2u\, u\,x^6 \japa^{2\theta-1}xdx
+\int \! u \partial_xu\,u\,x^6 \japa^{2\theta-1}dx=0.
$$

From \eqref{e000} and the hypothesis we have that after integrating in the time interval $[t_1,t_2]$ the contribution of the first and third terms above are bounded for $\theta\in(0,1/2)$. So it remains to consider the second one. Using \eqref{Id1c} we write
$$
\begin{aligned}
\int \hil\partial_x^2u&\, u\,x^6 \japa^{2\theta-1}xdx=
\int \hil\partial_x(x^3u)\, x^3u\, \japa^{2\theta-1}xdx\\
&\hskip10pt +6\int \hil\partial_x^2(x^2u)\, x^3u\, \japa^{2\theta-1}xdx
+6\int \japa^{2\theta-1} x\, \hil (xu)\, u\,x^3 dx\\
&\hskip10pt +\big(\int u(x,t)dx\big) \int \,x^3u\, \japa^{2\theta-1}x\,dx\\
&=L_1(t)+L_2(t)+L_3(t)+L_4(t).
\end{aligned}
$$

By hypothesis we have the $L_4(t)\equiv 0$ and a familiar argument shows that the  terms $L_1(t)$ and $L_2(t)$ can be written as a bounded term plus a multiple of
$$
\int (D^{1/2}(x^3u))^2\,\japa^{2\theta-1}\,dx.
$$
To handle $L_3(t)$ we write:
\begin{equation}
\label{term-c}
|L_3(t)|\leq c\|\japa^{2\theta-1} x \hil (xu)\|_2\| u\,x^3\|_2,
\end{equation}
which is bounded only for $\theta\in (0,1/4)$, see \eqref{A_2-precise} and \eqref{e00}. Collecting the above information and following our argument we can conclude that
\begin{equation}
\label{term-d}
|x|^{3+\beta}u\in L^{\infty}([t_1,t_2]:L^2(\R)),\;\;\beta\in(0,1/4).
\end{equation}
Reapplying the above argument using \eqref{e000} instead of \eqref{e00} one obtains that the term in \eqref{term-c} 
is bounded for $\theta\in (0,3/8)$. This iteration allows to extends the result to $\alpha\in(0,1/2)$. Notice that by Theorem A part (4) the result for $\alpha=1/2$ fails.

The rest of the proof in this case follows a similar argument to the one provided in details in the {\bf Case 1}, so it will be omitted.
\smallskip

\end{proof}


\end{document}